\documentclass[12pt]{article}

\usepackage[sumlimits, intlimits]{amsmath}
\usepackage{amssymb,amsthm}
\usepackage[mathscr]{eucal}

\newcommand{\R}[1]{\mathbb{R}^{#1}}
\renewcommand{\d}{\mathrm{d}}
\newcommand{\p}{\partial}
\newcommand{\tF}{\mathcal{F}}

\newtheorem{theorem}{Theorem}
\newtheorem{statement}{Proposition}
\newtheorem{remark}{Remark}
\newtheorem{lemma}{Lemma}
\newtheorem{consequence}{Corollary}
\newtheorem{definition}{Definition}
\title{
\bigskip
\bigskip
On Selfadjoint Subspace of One-Speed Boltzmann Operator}
\author{Roman Romanov\thanks{The was partially suppported by INTAS Grant 05-1000008-7883 and RFBR Grant
06-01-00249. }\\[10pt] Mihail Tihomirov}
\date{}
\begin{document}
\maketitle
\section{Introduction}

\hspace{\parindent} It is well-known \cite{Na,N} that a nonself-adjoint operator in a Hilbert space can be represented as an orthogonal sum of a self-adjoint one, and an operator having no reducing subspaces on which it induces a self-adjoint operator. A natural question about operators arising in applications is whether the first (selfadjoint) component in this sum is trivial, that is, whether the operator is completely nonself-adjoint. For differential Schr\"{o}dinger operators this question was studied earlier \cite{Pav1} and is related to the unique continuation property for solutions. In this note we study complete nonself-adjointness for one-speed Boltzmann operator \cite{JLh} arising in the theory of neutron transport in a medium with multiplication.

The main result of the paper -- theorem 2 -- is that the selfadjoint subspace  is non-trivial for any Boltzmann operator with polynomial collision integral if the multiplication coefficient has a lattice of gaps in the support of arbitra\-ri\-ly small width, that is, if the coefficient vanishes on an $ \varepsilon $-neighborhood of the set $ a \mathbb{Z} $ for some
$ a , \varepsilon > 0 $. On the other hand, the operator of the isotropic problem turns out to be completely nonself-adjoint if the multiplication coefficient is non-zero on a semi-axis (proposition \ref{halfax}). For anisotropic problem we give an example (corollary \ref{half}) showing that under an appropriate choice of the collision integral the operator may turn out to be completely nonself-adjoint for any non-vanishing multiplication coefficient. Finally, for the three-dimensional Boltzmann operator we establish non-triviality of the self\-adjoint subspace for any non-zero multiplication coefficient (theorem \ref{3dim}).

Let us describe the structure of the paper. Proposition 1 gives a version of the abstract theorem on decomposition of an operator in the sum of selfadjoint and completely nonself-adjoint ones convenient for our purposes. A close assertion in terms of the resolvent is contained in \cite{N}. Theorem 2 is proved by a direct construction of a non-zero function lying in the self-adjoint subspace. It occupies sections 3 and 4. The same problem for the three-dimensional Boltzmann operator is studied in section 5.

The authors are indebted to P. Kargaev for a useful discussion.

The following notation is used throughout: 
\begin{itemize}
\item If $\big\{S_i\big\}_{i\in I}$ is a family of subsets of a Hilbert space, then $\bigvee\limits_{i\in I} S_i$ is the closure of the linear span of the set $\bigcup\limits_{i\in I} S_i$.
\item  If $f\in L^2(\R{})$, then $\Hat{f}\in L^2(\R{})$ is the Fourier transform of $f$:
\[
\Hat{f}(p) \overset{\mathrm{def}}{=}
       \frac{1}{\sqrt{2\pi}} \int_{-\infty}^{\infty}e^{-ipx}f(x)\,\d x .
\]
\item $H^2_\pm $ --- the Hardy classes of analytic functions in the upper and lower half planes, respectively.
\item The abbreviation a. e. refers to the Lebesgue measure on $\mathbb{R} $. For a measurable function $f$ on $\R{}$ the notation $\mathrm{supp}f $ stands for the set
$\left\{x\in\ \mathbb{R}:\; f(x)\ne 0\right\}$ defined up to a set of zero measure.
\end{itemize}

\section{Definitions and Preliminaries}

\hspace{\parindent} The Boltzmann operator acts in the space
$L^2(\R{}\times[-1,1])$ of functions $ u = u ( x,\mu)$ ($ x \in \R{} $,
$\mu \in [ -1 , 1 ] $) endowed with the standard Lebesgue measure, by the formula:

\begin{equation}
( Lu ) ( x , \mu ) = i\mu \left( \p_x u \right) ( x , \mu ) + i
\sum_{\ell=1}^{n}c_{\ell}(x)\varphi_{\ell}(\mu) \int_{-1}^{1} u (
x , \mu^\prime ) \overline {\varphi_\ell (\mu')}\,\d\mu'.
\label{Bo}
\end{equation}
Here the functions $c_{\ell}\in L^{\infty}(\R{})$, $\varphi_{\ell}\in L^\infty
[-1,1]$, $\ell = 1,...,n$ are known pa\-rameters of the problem. The functions $c_{\ell}$ are assumed to be real-valued. The case $ n = 1 $, $ \varphi_1 \equiv 1 $ (isotropic scattering) is of special interest. In this si\-tuation the function $ c_1 $ is called the local multiplication coefficient, and the index $ 1 $ is omitted. Without loss of generality, one assumes throughout that the functions $\left\{\varphi_{\ell}\right\}_{\ell=1}^{n}$ are linearly independent.

Under the conditions imposed the operator $ L $ is the sum of the operator $ L_0 = i\mu\p_x $, selfadjoint on its natural domain, and a bounded one  (see \cite{JLh,Shikhov} for details). According to the non-stationary Boltzmann equation, if $ u^t $ is the particle density at time $ t $, then $ u^t = e^{ itL } u^0 $. Notice that the literature on the Boltzmann operator uses for it an expression different from (\ref{Bo}) by the factor $ i $, because of a different definition of the exponential function. The evolution operator $ u^0 \mapsto u^t $ in our notation coincides with the standard one. 

Let $ D $ be an operator in a Hilbert space $H$ of the form $ D
= A + i K $, where $ A $ is selfadjoint, and $ K $ is selfadjoint and bounded.

\begin{definition} The subspace $H_0 \subset H $ is the selfadjoint subspace of the operator  $D$, if
\begin{enumerate}
\item $H_0$ reduces\footnote{that is, the orthogonal projection on $ H_0 $ in $ H $ preserves the domain $\cal D $ of the operator $ D $, and $ D f \in H_0 $, $ D^* f \in H_0 $ for all $ f \in H_0
\cap {\cal D} $.} $D$, and the restriction $\left. D\right|_{H_0}$ is a selfadjoint operator in $H_0$;
\item any reducing subspace $H^\prime $ of the operator $D$ such that 
$\left. D\right|_{H^\prime}$
is a selfadjoint operator in $H^\prime $ is contained in $ H_0$.
\end{enumerate}

An operator $ D $ is called completely nonself-adjoint if its selfadjoint subspace is trivial.
\end{definition}

\begin{statement}\label{th1} The orthogonal complement of the selfadjoint subspace of the operator $D$ coincides with the subspace
\[
H_1\overset{\mathrm{def}}{=}\bigvee_{t\in\R{}}e^{iAt}\mathrm{Ran}\,K.
\]
\end{statement}

\begin{proof} Let $H_0$ be the selfadjoint subspace of the operator $D$. By definition, the subspace $H_1$, and hence $ H_1^\perp $, reduces the operator $ A $. Since, obviously, $H_1^\perp
\subset\mathrm{Ker}\, K $ we obtain from this that $H_1$ and $ H_1^\perp
$ are reducing subspaces of the operator $ D $ such that the restriction of $ D $ to $ H_1^\perp $ is a selfadjoint operator. Thus, $H_1^\perp \subset H_0 $. Let us show that $ H_1 \subset
H_0^\perp $. Indeed, the subspace $ H_0 $ reduces
$ D $, and therefore the operator $ A = \left( D + D^*
\right)/2 $ as well. This means, in particular, that $
e^{iAt}\mathrm{Ran}\, K \subset H_0^\perp $ for all real
$ t $, since $ H_0\subset\mathrm{Ker}\, K $.
\end{proof}

In what follows we are going to use the fact that the selfadjoint subspace $ H_0 $ is reducing for the operator $ A $ as well, and the selfadjoint part of $ L $ coincides with the restriction of $ A $ to $ H_0 $.

For the Boltzmann operator (\ref{Bo}) $ A = L_0 = i\mu\p_x $, and a straightforward calculation gives
\begin{equation}
\label{exp1d}
 \big(e^{itA}f\big)(x,\mu)=f(x - \mu t,\mu).
\end{equation}

For a function $ \xi \in L^{\infty}(\R{}) $ let us denote by $ \mathcal{D}_\xi   $ the set of compactly supported functions $ h \in L^2 ( \mathbb{R} ) $ vanishing outside $\mathrm{supp}\, \xi $.

\begin{consequence}
\label{l1} A function $ f\in L^2(\R{}\times[-1,1]) $ belongs to the selfadjoint subspace of the operator (\ref{Bo}) if and only if
\begin{equation}\label{cond}
\int_{[-1,1]\times \mathbb{R}} f(x-\mu
t,\mu){\overline{\varphi_{\ell}(\mu)}} h ( x ) \,\d\mu \d x = 0
\end{equation}
for all $\ell = 1,...,n$, $t\in\R{}$, and $ h \in
\mathcal{D}_{c_{\ell}}$.
\end{consequence}

Thus, we have to find out if there exists a non-zero function
$ f $ satisfying the condition (\ref{cond}). Let us first explain on the formal level the method we use. For simplicity, let the function $ c $ be the indicator of an interval $ I $, and $ \varphi \equiv 1 $. Then the condition of the lemma means that \[
\int_{-1}^1 f(x-\mu t,\mu)\,\d\mu = 0
\] for all $ t\in\R{} $ and $ x \in I $. We will search for the function $ f $
in the form \[ f ( x , \mu )  = \int_{\mathbb{R}} e^{i\tfrac{xq}\mu} u
( q,\mu )\,\mathrm{d}q . \] Substituting and interchanging the order of integrations, we obtain: \[ 0 = \int_{\mathbb{R}}\d q\, e^{-iqt}
\int_{-1}^1 e^{i\tfrac{xq}\mu} u(q,\mu)\,\mathrm{d}\mu . \] Since this equality is an identity in $ t $, the inner integral must be zero for all $ q $. After the change of variable $ p
= 1/\mu $ in this integral we arrive at the following uniqueness problem for the Fourier transform: is there a nonzero function $ v ( q , p ) $ such that its Fourier transform in the second variable $ \mathcal{F} v $ vanishes at all points of the form $ ( q , -xq )$, $ q \in \mathbb{R} ,\, x \in I $.
A rigorous argument requires analysis of certain integral transforms of the Fourier type, definitions and elementary properties of which are given in the next section.

\section{Integral Transforms}

\hspace{\parindent} Let $ \omega \subset \mathbb{R} $ be a compact interval. We set the notation for certain classes of functions of variables  $q\in \mathbb{R}$ and $\mu\in[-1,1]$ and transforms between them:
\begin{itemize}
\item $H = L^2(\R{}\times[-1,1]) $.
\item $\overset{\circ}{C}_\omega $
 --- the linear set of functions $u\in  C_0^{\infty}(\R{}\times
[-1,1])$ vanishing for $ q \notin \omega $.
\item $ H_\omega $ --- the subspace in $ L^2\big(\mathbb{R}
\times [-1,1],|\mu|\,\d q \, \d\mu\big) $ of functions vanishing for $ q \notin \omega $.
\item $ \Phi $ and $ \Phi^\ast $ --- unitary mutually inverse operators \[ \Phi : H \rightarrow
L^2(\R{}\times[-1,1],|\mu|\mathrm{d}q \, \mathrm{d}\mu) , \]
\[ \Phi^\ast :
L^2(\R{}\times[-1,1],|\mu|\mathrm{d}q \, \mathrm{d}\mu) \rightarrow H,  \]
defined on finite smooth functions by formulae \[
\big(\Phi f\big)(q,\mu) = \frac{1}{\sqrt{2\pi}}\int_{\mathbb{R}}
                          e^{-i xq\mu}f(x,\mu)\,\mathrm{d}x,
\]
\[
\big(\Phi^{\ast} u\big)(x,\mu) = \frac{1}{\sqrt{2\pi}}\int_{\R{}}
                          e^{i\tfrac{xq}{\mu}}u(q,\mu)\,\mathrm{d}q .
\]
\end{itemize}

\begin{remark}\label{eqvi}
It is obvious that $\Phi\big(-i\mu\p_x\big)\Phi^{\ast}$ is the operator of multiplication by the independent variable $q$ in $L^2(\R{}\times[-1,1],|\mu|\mathrm{d}q \, \mathrm{d}\mu)$.
\end{remark}

Let us denote by $ \tF $ the unitary operator $ \tF \colon
L^2(\R{2}) \to L^2(\R{2}) $ of the Fourier transform in the second variable:
\[
\big(\tF u\big)(q,p)= \frac 1{\sqrt{2\pi}}\int_{\mathbb{R}}
                          e^{-isp}u(q,s)\,\mathrm{d}s .
\]

We are going to use the change of variables \[ ( J v ) ( q , s ) :
=
\begin{cases} s^{-2} v \left( q , s^{ -1} \right), & |s| > 1 \cr 0, &
|s| < 1 , \cr \end{cases} \] which defines an isometry \[ J \colon
L^2(\R{}\times[-1,1],|\mu|\mathrm{d}q \, \mathrm{d}\mu) \to
L^2(\R{2},| p|\mathrm{d}q \, \mathrm{d}p) . \]

For any smooth finite function $v $ defined in the strip $ \R{}\times
[-1,1] $ let
\begin{equation}
\label{psi} \big(\Psi v\big)(q,x) = \frac{1}{\sqrt{2\pi}}\int_{|p|
> 1 }  e^{ixqp} p^{ -2 } v\left( q,\frac 1p \right)\,\mathrm{d}p .
\end{equation}

\begin{lemma} For any closed interval $ \omega $ not containing $ 0 $, the transform defined by the formula (\ref{psi}) on $
\overset{\circ}{C}_\omega $ is extended to a bounded operator $\Psi$ from $H_\omega $ to $L^2(\R{2})$ acting by the following formula:
\begin{equation} \label{PsH} \big(\Psi v\big)(q,x) = ( \tF
J v) ( q , -xq ) , \; \; v \in H_\omega .
\end{equation}
\end{lemma}

\begin{proof} By definition (\ref{psi}), the equality
(\ref{PsH}) is satisfied for all $ v \in \overset{\circ}{C}_\omega
$, and the Fourier transform $ \tF $ in it can be understood classically. Then, for any $ g \in L^2(\R{2}) $ vanishing when $ |q| < a \colon = \mbox{dist} ( 0 ,
\omega ) $, we have:
\[
\int_{\R{2}} \left| ( \tF g ) ( q , -xq ) \right|^2 \mathrm{d}q \,
\mathrm{d}x = \int_{\mathbb{R}} \mathrm{d}q \frac 1{|q|}
\int_{\mathbb{R}}  \left| ( \tF g ) ( q , x ) \right|^2
\mathrm{d}x \le \frac{1}{a}\int_{\R{2}}  \left| g ( q , x )
\right|^2 \mathrm{d}x \, \mathrm{d}q .
 \]
Thus, the map $ \Psi $ is a composition of the bounded operator $ J \colon H_\omega \to L^2(\R{2}) $ and the map $ g
\mapsto (\mathcal{F}g)( q , -xq ) $, which is a bounded operator from the sub\-space $ J H_\omega \subset L^2(\R{2}) $ to $
L^2(\R{2}) $. Since the linear set $\overset{\circ}{C}_\omega $ is dense in $H_\omega $, it follows that the map $ \Psi $ defines a bounded operator.  Simultaneously, we have proved (\ref{PsH}).
\end{proof}

\begin{lemma}\label{qwerty}
Let $ \omega $ be a closed interval not containing $ 0 $, and let $ \varphi \in L^\infty ( -1 , 1) $. Then for any compactly supported function $ h \in L^2 (\mathbb{R}) $ and any 
$u \in H_\omega $ the following equality is satisfied for all $ t \in \mathbb{R}$: \begin{equation}
\label{ggg} \int\limits_{[-1,1]\times \mathbb{R}} \d\mu\, \d x \,
\overline{h ( x )} \varphi(\mu)\big(\Phi^{\ast} u\big)(x-\mu
t,\mu) = \int_{-\infty}^{\infty}\d q\, e^{-iqt}\left\langle \Psi
(u \varphi )( q,\cdot ), h \right\rangle_{ L^2 ( \mathbb{R} ) } .
\end{equation} 
\end{lemma}

\begin{proof}
Since $\overset{\circ}{C}_\omega $ is dense in $ H_\omega $, it is enough to prove (\ref{ggg}) for arbitrary func\-tion
$u\in\overset{\circ}{C}_\omega $. We have:
\begin{eqnarray*}
\lefteqn{\int_{-1}^{1} \varphi(\mu) \big(\Phi^{\ast}u\big)(x-\mu
t,\mu)\d\mu
 =  \frac{1}{\sqrt{2\pi}} \int_{-1}^{1}\d\mu \, \varphi(\mu)\int_{-\infty}^{\infty}
      e^{i\tfrac{(x-\mu t)q}{\mu}}u(q,\mu)\,\mathrm{d}q =}\\
& = & \int_{-\infty}^{\infty}\d q\,
e^{-iqt}\frac{1}{\sqrt{2\pi}}\int_{-1}^1
      e^{i\tfrac{xq}{\mu}}u(q,\mu) \varphi(\mu)\,\mathrm{d}\mu=
   \int_{-\infty}^{\infty}\d q\, e^{-iqt}\big(\Psi (u \varphi )\big)(q,x).
\end{eqnarray*}
Multiplying this equality by $ \overline h $ and integrating in $ x $, we obtain (\ref{ggg}). The interchange of integrations in $ x $ and $
q $ in the right hand side is possible because of the function $ h $ having a compact support.
\end{proof}

\section{Conditions of Complete Nonself-Adjointness}

In what follows, the functions $\varphi_{\ell}\in L^\infty [-1,1] $ from the definition of the Boltzmann operator are supposed to be extended by zero to the whole of the real line.

\begin{theorem}\label{th2}
Let the function $ F(q,p) \in L^2(\R{2},|p|\d q\d p)$ satisfy the following conditions:
\begin{eqnarray}
& \label{Lmu} F ( q , p ) = 0 \; \text{for} \;  |p| < 1 ; &
\\
\label{criterium} & \int_{\mathbb{R}} e^{ipxq}
F(q,p)\overline{\varphi_{\ell}(\tfrac{1}{p})}\,\d p = 0 &
\end{eqnarray}
for all $\ell = 1,...,n,$ and a. e. $q\in\R{}$,
$x\in\mathrm{supp}\,c_{\ell} $. Then the vector
\begin{eqnarray} \label{expli}
& f = \Phi^\ast u , & \\ \nonumber
& u ( x , \mu ) \colon = \mu^{-2} F\left( q, \mu^{ -1 } \right),
\; \; | \mu | < 1, &
\end{eqnarray} belongs to the selfadjoint subspace $ H_0 $ of the operator (\ref{Bo}).
The mapping $ F \mapsto f $ defines an isomorphism of the subspace
$ X \subset L^2(\R{2},|p|\d q\d p)$ singled out by conditions
(\ref{Lmu}) and (\ref{criterium}), and the space $ H_0 $. In particular, the space $ H_0 $ is non-zero if, and only if, there exists a non-zero function $ F $ satisfying conditions (\ref{Lmu}) and (\ref{criterium}).
\end{theorem}

\begin{remark}
The equality (\ref{criterium}) is understood as a condition of vanishing of the Fourier transform in the second variable of the function $
F(q,p)\overline{\varphi_{\ell}( p^{-1})}$ lying in the space $L^2(\R{2})$, on the set $ M_\ell \equiv \{ ( q ,
-xq ): \; q \in \mathbb{R} , \, x \in \mathrm{supp}\,c_\ell \} $ of positive planar Lebesgue measure.
\end{remark}

\begin{proof}
Substituting $ p = \mu^{ -1 } $, we immediate\-ly verify that the map $ F \mapsto f $, defined in the theorem, is an isometry from $ X $ to $ H $. Let us show that $ f \in H_0
$ for any $ F $ from the dense in $ X$ linear set of functions $ F \in X $  such that $ F(q,p) = 0 $ when $ q \notin \omega $ for some closed interval $ \omega = \omega ( F ) $ not containing $ 0 $. For such $ F $'s the function $ v
=u\overline{\varphi}_\ell $ obeys the equality (\ref{PsH}):
\begin{equation} \label{Psf} \Psi \left[
u\overline{\varphi}_\ell \right] ( q , x ) = \mathcal{F} \left[
F\left( q, p \right) \overline{\varphi_\ell \left( p^{ -1 }
\right) } \right] ( q , -xq ).
\end{equation}
By assumption (\ref{criterium}), the right hand side vanishes on the set $ \{ ( q , x ): \; q \in \mathbb{R} , \, x \in \mathrm{supp}\,c_{\ell} \} $, and thus $ \Psi \left[ u\overline{\varphi}_\ell \right] ( q
, x ) h ( x ) $ is identically zero for any function $ h \in L^2
( \mathbb{R} ) $ supported on $ \mathrm{supp}\,c_{\ell} $.
Applying the identity (\ref{ggg}), we conclude from this that
\[\int\limits_{[-1,1]\times \mathbb{R}} \d\mu\d x \,
\overline{h ( x ) \varphi_\ell (\mu)}\big(\Phi^{\ast} u\big)(x-\mu
t,\mu) = 0
\]
for all $ t\in\R{}$ and $ h \in \mathcal{D}_{c_{\ell}}$, that is, $f$ satisfies the condition of corollary \ref{l1}.

It remains to check that the range of the map $ F \mapsto f $ is the whole of $ H_0 $. Let $ f_0 \in H $ be a vector of the form $ f_0 =
P_\omega g $, where $ g \in H_0 $, $ \omega $
is a closed interval not containing point $ 0 $, and $ P_\omega $ is the spectral projection of $
L_0 $ corresponding to the interval $ \omega $. Then, $ f_0 \in H_0 $ since the subspace $H_0$ reduces the operator $ L_0 $, and hence any of his spectral projections. We shall show that $ f_0 $ lies in the range of the constructed isometry from $ X $ to $ H $.

Let $ u = \Phi f_0 $, and let $ F(q,p) = p^{ -2 } u(q,p^{ -1
})$ for $ q \in \omega $ and $|p| > 1$, $ F(q,p) = 0 $ for any other $(q,p)\in\R{2}$. By construction, the function $ F $ belongs to
$ L^2(\R{2},|p|\d q\, \d p)$ and satisfies (\ref{Lmu}) and
(\ref{expli}) with $ f = f_0 $. Notice that the function $ u $ vanishes when $ q \notin \omega $ since, according to remark \ref{eqvi}, $\Phi L_0\Phi^{\ast} $ is the operator of multiplication by the $q$ variable. Thus, the function $u
\in H_\omega $, lemma \ref{qwerty} applies to it because $ 0 \notin \omega $, and the equality (\ref{Psf}) holds true. As the function $f_0$ belongs to $ H_0 $, and hence obeys condition
(\ref{exp1d}), the left hand side in (\ref{ggg}) vanishes for the $ u $ under consideration for all $ t\in\R{}$ and $ h \in
\mathcal{D}_{c_{\ell}}$, $ \ell = 1, \dots ,n $. By uniqueness of the Fourier transform it follows that \[ \left\langle \Psi
\left[ u\overline {\varphi_\ell} \right]( q,\cdot ), h
\right\rangle_{ L^2 ( \mathbb{R} ) } = 0 \] for a. e. $ q \in
\mathbb{R} $ and all $ h \in \mathcal{D}_{c_{\ell}}$. The arbitrariness of $ h $ implies that the function $ \Psi \left[
u\overline{\varphi}_\ell \right] ( q , x ) $ vanishes on 
$ \mathbb{R} \times \mathrm{supp}\, c_{\ell}$. Then the right hand side in (\ref{Psf}) also vanishes on $
\mathbb{R} \times \mathrm{supp}\, c_{\ell}$, that is, condition (\ref{criterium}) is satisfied. It remains to notice that the set of vectors $ f_0 $ of the form under consideration is dense in $ H_0 $ since the operator $ L_0 $ is absolutely continuous. \end{proof}

\begin{theorem}
Let the function $c(x)$ be bounded, and let there be $ a,\varepsilon
> 0 $ such that $ c ( x ) = 0 $ for $ | x - x_0 - aj | <
\varepsilon $, $ j \in \mathbb{Z} $, with some $ x_0 \in
\mathbb{R} $, and let all the functions $ \varphi_\ell ( \mu ) $, $ 1
\le \ell \le n $, be polynomials. Then the selfadjoint subspace $ H_0 $ of the Boltzmann operator
\begin{equation}\label{Both2} L = i\mu \p_x + i c(x)
\sum_{\ell=1}^{n}\varphi_{\ell}(\mu) \int_{-1}^{1} \cdot \;
\overline {\varphi_\ell (\mu')}\,\d\mu'  \end{equation}
is non-trivial, and, moreover, the restriction of the selfadjoint part of the operator $ L $ to its spectral subspace corresponding to the interval $ [ - \pi / a  , \pi / a ] $ has Lebesgue spectrum of infinite multiplicity\footnote{This means that the restriction is unitarily equivalent to an orthogonal sum of infinitely many copies of the operator of multiplication by the independent variable in $ L^2
$ over this interval.}.
\end{theorem}

\begin{proof}
Without loss of generality one can assume that $ x_0 = 0 $ and $ n = 1 + \max_\ell \deg \varphi_\ell $. Let us search for a function $ F $, satisfying the conditions of theorem \ref{th2}, in the form $ F ( q , p ) =
\chi ( q ) f ( p q ) $, where $ \chi $ is an arbitrary bound\-ed function on the real axis such that $ \mathrm{supp}\, \chi = [ - b , b ] $ for some positive $ b < \pi / a $, and $ \chi ( q ) / q
\in L^2 $. The conditions of Theorem 1 will be met if the function $ f \in
L^2 ( \mathbb{R} , |p| \d p ) $ obeys the following requirements:

(\textrm{i}) $ f ( p ) = 0 $ for $ | p | \le b $;

(\textrm{ii}) $ \mathrm{supp}\, \widehat{f p^{ -j }} $ is contained in the 
$ \varepsilon $-vicinity of the set $ a \mathbb{Z} $ for all $
j $, $ 0 \le j \le n - 1 $.

We are going to use the following observation: let $ h \in L^2_{ loc } (
\mathbb{R} ) $ be an arbitrary $ 2 \pi $-periodic function, and
$ \omega $ be a smooth function on the real line supported on an interval $ ( - \delta , \delta ) $, $ \delta > 0 $. Then the function $ \xi = h \hat{\omega} $, obviously, belongs to $ L^2 (
\mathbb{R} , | p | \d p ) $,  and its Fourier transform vanishes outside the $ \delta $-vicinity of $ \mathbb{Z} $. This observation follows from elementary properties of convolution since $ h = \hat{ \rho } $ where $ \d \rho = \sum_j \rho_j
\delta ( x - j ) $ is the discrete measure with masses being the Fourier coefficients $ \rho_j $ of the restriction of the function $ h $ to a period.

Fix an arbitrary nonzero function $ h $, satisfying the conditions above and such that $ h ( p) = 0 $ for $ | p | \le \pi
- \nu $, $ \nu > 0 $. Let $ \delta = \varepsilon/a $, choose an arbitrary nonzero function $ \omega_0 \in C_0^\infty
( \mathbb{R} ) $ supported on $ ( - \delta , \delta )
$ and define the corresponding function $ \xi $ setting $ \omega =
\omega_0^{ (n) } $. Define $ f ( p ) = \xi ( a p ) $. By construction, conditions (\textrm{i}) and (\textrm{ii}) hold true for the function $ f $ for all $ \nu
> 0 $ small enough. Fix such a $ \nu $ and let:
\[
u (q,\mu) = \mu^{ -2 } F \left(q, \frac{1}{\mu} \right) .
\]
By theorem 1 the nonzero function $ g \overset{\mathrm{def}}{=} \Phi^{\ast} u
$ belongs to $ H_0 $, and the non-triviality of the subspace $ H_0 $ is proved.

To establish the assertion about the multiplicity of the spectrum notice that, as follows from remark \ref{eqvi}, the restriction of $ L_0 $ to its reducing sub\-space generated by the function $ g $ is unitarily equivalent to the operator of multi\-pli\-cation by the independent variable in the space $ L^2 $ over the support of $ \chi $, that is, in $ L^2 (  - b ,
b ) $. Each choice of the function $ h $ in the construction above then corresponds to a reducing subspace, and if continuous functions $ h_j $, $ j = 1, \dots , N $, $ N < \infty $, are mutually linearly independent, then so are the corresponding reducing subspaces $ Y_j \subset H_0 $. Indeed, the last assertion means that for any finite $
M $, any $ h_j $ satisfying the conditions above, and any $
\chi_j \in L^2 (  - b , b ) $, $ j \le N $, the following implication is true:
\[ \sum_1^M h_j ( p q ) \hat{\omega} ( pq ) \chi_j ( q ) \equiv 0
\Rightarrow \chi_j ( q ) \equiv 0 \, \forall j \le N ,
\] which is easily verified by induction. It is then enough to choose an arbitrary $ c \ne 0 $ such that $ \hat{ \omega } ( c ) \ne 0 $, and $ h_j ( c ) \ne 0 $ for at least one $ j $, and let $ p = c/ q $.

Thus, we have proved that for any $ b < \pi / a $ there is a reducing subspace in $ H_0 $ such that the restriction of the operator to it has Lebesgue spectrum of infinite multiplicity on $ [ -b , b ] $, hence the same is true of $ b = \pi / a $. Since the operator $ L_0 $ is absolutely continuous, it follows that the restriction of $ L $ to $ H_0 $ possesses the same property. \end{proof}

\begin{remark}
The proof of theorem 2 is constructive -- nonzero vectors from $ H_0
$ were found explicitly.
\end{remark}

Sometimes it is possible to say more about the spectrum of the selfadjoint part.

\begin{statement}\label{compact}
Let the function $ c ( x ) $ be compactly supported, and let all the functions $ \varphi_\ell ( \mu ) $, $ 1 \le \ell \le n $, be polynomials. Then the selfadjoint part of the operator $ L $ of the form
(\ref{Both2}) is unitarily equivalent to an orthogonal sum of infinitely many copies of the operator of multiplication by the independent variable\footnote{Theorem 2 in the situation under consideration only ensures the existence of the spectrum in a vicinity of $ 0 $.} in $ L^2 ( \mathbb{R} ) $. 
\end{statement}

\begin{proof} Let us first consider the case $ n = 1 $, $ \varphi_1 \equiv 1 $.
Let $ I $ be an arbitrary closed interval, $ \chi ( q ) $ its indicator function. We shall search for the function $ F ( q , p ) $ in the form of the product $ \chi ( q ) f ( p ) $ where $ f  \in L^2 ( \mathbb{R}
, |p| \d p ) $ is a function vanishing on $ [ - 1 , 1 ] $ and such that $ \hat{f} $ vanishes on an interval $ [ -M, M ] $. It is clear that for $ M $ large enough such a function $ F $ obeys all the conditions of theorem \ref{th2}. A supply of functions $ f $ with the desired properties is provided by the following lemma.

\begin{lemma}
\label{example} Let $\alpha > 0$, and let $ \rho ( z ) $ be an arbitrary nonzero function analytic in the plane cut along a compact interval $ J \subset (-\infty, -1] $ and such that $
\rho ( z) = O \left( \left| z \right|^{-2 } \right) $ when $ |z|
\to \infty $ uniformly in $ \mathrm{arg}\, z $, and the restrictions of $ \rho
( z ) $ to $ \mathbb{C}_\pm $ belong to $ H^2_\pm $, respectively. Define the function 
\begin{equation} \label{Falp}
  \phi_\alpha (z) = \exp\left[2i\alpha\left(-\frac{z}{2} +
 \frac{1}{ 1-\sqrt{\frac{z-1}{z+1}} }\right)\right] \rho ( z ) ,
\end{equation}
where the branch of the square root is chosen so that $ \phi_\alpha
( z ) $ be analytic in the plane cut along the rays $ (-\infty,
-1]\cup [ 1, +\infty) $, and $\mathrm{Im}\sqrt{\frac{z-1}{z+1}}
> 0 $. Let $ f_\alpha^\pm $ be the boundary values of the function $ \phi_\alpha $ on the real axis in the sense of the Hardy classes.

Then the (obviously, nonzero) function
\[
f_\alpha (x) \overset{\mathrm{def} }{=}
   f_\alpha^+ (x) - f_\alpha^- (x)
\]
obeys:
\begin{enumerate}
\item  $f_\alpha \in L^2(\R{},|p|\d p) $;
\item  $f_\alpha (x) = 0$ for $|x|<1$;
\item  $\Hat{f}_{\alpha}(p) =  0$ for $ |p| \le\alpha $.
\end{enumerate}
\end{lemma}

\begin{proof}
Property 2 is obvious. Since the boundary values of the exponent in (\ref{Falp}) on the real axis have the modulus $ \le 1 $ for the given choice of the square root brunch, the inclusion $f^{\pm}_\alpha \in L^2(\R{}, |p| \d
p) $ is immediate from the assumptions about the function $ \rho ( z ) $. It remains to check the property 3.

The following asymptotics hold for $\lvert z\rvert\to\infty$ in each of the halfplanes $
\mathbb{C}_\pm $ uniformly in $ \mathrm{arg}
\, z $:
\begin{eqnarray*}
 \frac{1}{ 1-\sqrt{\frac{z-1}{z+1}} } =\frac{1}{
1-\left(1-\frac{1}z+ O\left(\frac{1}{z^2}\right)\right) }= z +
O(1),\; \; \text{for} \;  \mathrm{Im}\, z
> 0; \\ \frac{1}{ 1-\sqrt{\frac{z-1}{z+1}} }=\frac 1{
1+\left(1-\frac{1}z+ O\left(\frac{1}{z^2}\right)\right) }= O(1),
\;\; \text{for}\; \mathrm{Im}\, z < 0 . &
\end{eqnarray*}
Therefore for $ \lvert z\rvert\to\infty $ we have:
\[
\exp\left[2i\alpha\left(-\frac{z}{2} +
 \frac{1}{ 1-\sqrt{\frac{z-1}{z+1}} }\right)\right]=
\exp\big(i\alpha z\,\mathrm{sign}(\mathrm{Im}\,z)+O(1)\big).
\]
Thus, the restrictions of the functions $ e^{\mp i\alpha z}\phi_\alpha $ to the halfplanes $ \mathbb{C}_{\pm} $ are in $ H^2_\pm $, respectively.
By the Paley-Wiener theorem this implies that
$\widehat{f^\pm_\alpha}(p)=0$ when $ \pm p\le\alpha $, hence $
\Hat{f}_\alpha (p ) = 0 $ for $ |p| \le \alpha $.
\end{proof}

For the function $ \rho $ in this lemma one can take, for instance, the branch of the function $ \ln^n\left(\dfrac{z+a}{z+b}\right) $, $ 1 < b < a
$, $ n \ge 2 $, analytic in the plane cut along the interval $ [-a, -b] $, fixed by the condition
${\left.\mathrm{Im}\ln{\frac{z+a}{z+b}}\right|}_{z=0}=0$.

\noindent {\it End of proof of proposition
\ref{compact}.} Let $ \alpha $ be a number such that $ | q x | < \alpha $ for all $ q \in I $, $ x \in \mathrm{supp} \, c $, 
$f_\alpha \in L^2 ( \mathbb{R} , |p| \d p ) $ an arbitrary function vanishing on   $ [ - 1 , 1 ] $ and such that $
\hat{f_\alpha} $ vanishes on the interval $ [ - \alpha , \alpha ] $.
Let $ F ( q , p ) = \chi ( q ) f_\alpha ( p ) $. Define a vector $ g \in H_0 $ via the function $ F $ in the same way as in the proof of theorem 2. The restriction of the operator $ L $ to its reducing subspace $
Y = Y( f_\alpha ) $, generated by the vector $ g $, is unitarily equivalent to the operator of multiplication by the independent variable in the space $ L^2 ( I ) $, and if functions $ f_{\alpha , j }
$, $ j = 1 , \dots , n < \infty $, are mutually linear independent, then so are the corresponding subspaces $ \left\{ Y( f_{\alpha , j } )
\right\} $. The assertion of the proposition now follows from this and the fact that the linear space of functions $ f_\alpha $ constructed in lemma \ref{example} is infinite-dimensional.

The general case ($ n \ne 1 $) is considered in a similar way, we only require additionally the function $ \rho $ in lemma \ref{example} to have a zero of order $ n - 1 $ at the point $ 0 $. If this requirement is satisfied, the Fourier transforms of $ f p^{ -j }
$ vanish on $ [ - \alpha , \alpha ] $ for all $ j \le
n - 1 $, and the proof proceeds as above.
\end{proof}

{\textit {Commentary to the proof of theorem 2.}} The question if there exists a nonzero function $ f \in L^2 (\mathbb{R}) $ such that the restrictions
$ \left. f \right|_S = 0 $ and $ \left. \hat{f}
\right|_\Sigma = 0 $ for a given interval $ S $ and a set $
\Sigma \subset \mathbb{R} $ is known as the Beurling prob\-lem and has been studied for a long time \cite{HJ}. For instance, the Amrein-Berthier theorem \cite{HJ} establishes the existence of such functions if the set $ \Sigma $ has fini\-te measure, the Kargaev theorem \cite{kargaev} -- in a situation generalizing theorem 2 to the case of gaps narrowing at infinity. These results are, however, not immediately applicable to the problem under consideration, when the func\-tion $ f $ is subject to an additional condition of square summability with the grow\-ing weight $ |p| $. To use them, one would have to smoothen up the func\-tions cons\-tructed which would lead to assertions close to theorem 2 and pro\-position
\ref{compact}, obtained here by elementary methods.

The selfadjoint subspace found in theorem 2 is quite large, and it is na\-tu\-ral to ask if there is much else. On this is the following

\begin{remark}
Results in paper \cite{KNR} show that in the isotropic problem the essential spectrum of the restriction of the operator $ L $ to $ H_0^\perp $
coincides with the real line if the function $ c $ is compactly supported, and     $ c ( x ) \ge 0 $ a. e.
\end{remark}

In the direction opposite to theorem 2 the following simple assertion holds.

\begin{statement} \label{halfax} Let the function $c \in L^\infty (\mathbb{R})$ be such that $ c ( x ) \ne 0 $ a. e. on a semi-axis. Then the Boltzmann operator 
\[
L= i\mu\p_x + i c(x) \int_{-1}^1 \cdot\,\d\mu'
\]
is completely nonself-adjoint.
\end{statement}

\begin{proof}
Without loss of generality one can assume that $ c ( x ) \ne 0 $ for a. e.
$ x > 0 $. Suppose that the selfadjoint subspace $ H_0
\ne \{ 0 \} $. Then by theorem \ref{th2} (see (\ref{criterium})) there exists a nonzero function $ F ( q , p ) \in L^2(\R{2})$ such that for a. e. $ q > 0 $ we have: $ ( \mathcal{F}^* F) ( q , x
) = 0 $ for a. e. $ x > 0 $. By the Paley-Wiener theorem this implies that
$F(q,\cdot)\in H^2_+ $ for a. e. $ q > 0 $, and, since $
F(q,p)=0 $ for $ |p| < 1 $, by properties of the Hardy classes it follows that the function $ F(q,\cdot) = 0 $ for a. e. $ q
> 0 $. Similarly, one considers the case $ q < 0 $. We thus obtain that $ F $ is the zero function, a contradiction.
\end{proof}

The following proposition is aimed at clarifying the main result of theorem 2. The operator in a strip of half-width dealt in it has possibly no physical relevance.

\begin{statement} Let $\varphi \in L^\infty ( 0 , 1 ) $,
$ \mathrm{supp} \, \varphi =  [ 0 , 1] $; $ c \in L^\infty (
\mathbb{R} ) $, and let $ L $ be an operator in the Hilbert space $ H = L^2(\R{}\times[0,1])$ defined by the formula
\begin{equation}
( Lu ) ( x , \mu ) = i\mu \left( \p_x u \right) ( x , \mu ) + i
c(x)\varphi (\mu) \int_0^1 u ( x , \mu^\prime ) \overline {\varphi
(\mu')}\,\d\mu'  \label{Bohalf}
\end{equation}
on a natural domain of its real part $ L_0
= i\mu \p_x $. Then the operator $ L $ is comp\-letely nonself-adjoint if $
c \not \equiv 0 $.
\end{statement}

\begin{proof}
Arguing as in the poof of theorem 1, it is easy to see that the operator defined by (\ref{Bohalf}) is completely nonself-adjoint if any function $F( q , p ) \in L^2(\R{2},|p| \d q \d p)$, satisfying the condition (\ref{criterium}) and such that $ F(q,p)=0 $ for $ p < 1 $, vanishes identically. The condition
(\ref{criterium}) means that for a. e. $ q \in \mathbb{R} $ the Fourier transform in the second variable of the function $ G ( q , p )
= F(q,p)\overline{\varphi \left( p^{ -1} \right)}$ vanishes on a set of positive measure. On the other hand, $ G ( q , p ) = 0 $ for $ p < 1 $. By properties of the Hardy classes this implies that the function $ G (q,\cdot) \equiv 0 $ for a. e. $ q \in \mathbb{R} $, and thus $ F \equiv 0 $.
\end{proof}

A similar assertion holds for the strip $ \R{}
\times[-1,0]$. Considering the ortho\-gonal sum, we obtain the following

\begin{consequence} \label{half}
Let the functions $ \varphi_{1,2} \in L^\infty ( -1 , 1 ) $ be such that $ \mathrm{supp}\,\varphi_1 = [ 0 , 1 ] $, $
\mathrm{supp}\,\varphi_2 = [ -1 , 0 ] $. Then the Boltzmann operator of the form
\[ L= i\mu\p_x + i
c_1(x) \varphi_1 ( \mu ) \int_{-1}^1 \cdot\,\overline{\varphi_1
(\mu')}\,\d\mu' + i c_2 (x) \varphi_2 (\mu)\int_{-1}^1
\cdot\,\overline{\varphi_2 (\mu')}\,\d\mu'
\]
is completely nonself-adjoint if neither of the functions $c_1 $, $ c_2 $ vanishes identi\-cally.
\end{consequence}

Thus, in the anisotropic case the Boltzmann operator may turn out to be completely nonself-adjoint for perturbations having arbitrarily small support.

\section{Three Dimensional Boltzmann Operator}

\hspace{\parindent} Let $ \mathbb{S}^2 = \{ s \in \mathbb{R}^3 :
\; |s| = 1 \} $. The 3D Boltzmann operator acts in the space 
$L^2(\R{3}\times \mathbb{S}^2)$ of functions $u =
u(x,\mu)$ ($x\in\R{3}$, $\mu\in \mathbb{S}^2$) by the formula:
\begin{equation}
( Lu ) ( x , \mu ) = i\mu \left( \nabla_x u \right) ( x , \mu ) +
i \sum_{\ell=1}^{n}c_{\ell}(x)\varphi_{\ell}(\mu)
    \int_{\mathbb{S}^2} u ( x , \mu^\prime ) \overline {\varphi_\ell (\mu')}\,\d S(\mu').
\label{Bo3}
\end{equation}

Here $c_{\ell} \in L^{\infty}(\R{3}) $ and $ \varphi_{\ell} \in L^2
( \mathbb{S}^2 ) $, $\ell = 1,...,n$, are known functions.
The operator $ L $ is a bounded perturbation of the operator $ L_0 =
i\mu\nabla_x $ selfadjoint on is natural domain.

\begin{theorem}\label{3dim}
The selfadjoint subspace of the Boltzmann operator (\ref{Bo3})
is non-trivial.
\end{theorem}

\begin{proof}  Let $ U: H \to H $ be the Fourier transform in the $ x $ variable. Let $ \hat L = U L U^* $, $ \hat L_0 = U L_0
U^* $ etc.

As in the 1D case, a vector $u$ belongs to the selfadjoint subspace of the operator (\ref{Bo3}) if, and only if
\[
\int_{\mathbb{S}^2} v(x-\mu
t,\mu)\overline{\varphi_{\ell}(\mu)}\,\d S(\mu) = 0
\]
for all $\ell = 1,...,n$, $t\in\R{}$, and a. e. $x\in\mathrm{supp}\,c_{\ell}$.
It is easy to see that this equality is satisfied if $\hat{v}:=Uv$ obeys
\begin{equation}
\label{bobobo}
\int_{\mathbb{S}^2} \exp\big(it\langle p,\mu\rangle_{\R{3}}\big)
 \hat{v}(p,\mu)\overline{\varphi_{\ell}(\mu)}\,\d S(\mu) = 0
\end{equation}
for a. e. $p\in\R{3} $ and all $t\in\R{}$, $\ell = 1,...,n$.

For each $p\in \R{3}$ define the spherical coordinates $(\psi_{p}, \theta_{p})$ on the sphere $\mathbb{S}^2$ of the
$ \mu $ variable choosing the polar axis aimed along the vector $ p $. Here $\theta_p $ and $\psi_p$ are the azimuthal and precession angles, respectively. Then, obviously, any smooth function $u\in L^2(\R{3}\times
\mathbb{S}^2)$ such that
\[
\int_{-\pi}^{\pi} u\big(p,\mu(\psi_{p}, \theta_{p})
\big)\overline{\varphi_{\ell}\big(\mu(\psi_{p},
\theta_{p})\big)}\,\d\psi_{p} = 0
\]
for a. e. $p\in\R{3}$, $\theta_{p}\in
[-\frac{\pi}{2},\frac{\pi}{2}]$, $\ell = 1,...,n$, satisfies (\ref{bobobo}), and hence $U^\ast u$ belongs to the subspace
$ H_0 $. \end{proof}

\begin{remark} The reducing subspace of the selfadjoint part of the operator constructed in the course of the proof, is, in general, a proper subspace in $ H_0 $.
\end{remark}

\pagebreak


\begin{thebibliography}{99}

\bibitem{Na}  B.~Sz\"{o}kefalvi-Nagy and C.~Foias, {\it Analyse
Harmonique des Operateurs de l$^{\prime }$Espase de Hilbert},
Masson et C$^{ie}$/ Academiai Kiado, 1967.

\bibitem{N} S.N.~Naboko, "A functional model of perturbation theory and
its applications to scattering theory", {\it Trudy MIAN} {\bf 147}
(1980), 86 - 114 ({\it Russian}); English transl. in: {\it Proc.
Steklov Inst. Math.} (1981), No. 2, 85 - 116.

\bibitem{JLh} J.~Lehner, "The spectrum of the neutron transport
operator for the infinite slab",  {\it J. Math. Mech.}, {\bf 11}
(1962), No.  2, 173--181.

\bibitem{Pav1} B.S.~Pavlov, "Selfadjoint dilation of the
dissipative Schr\"odinger operator and its resolution in terms of
eigenfunctions", {\it Mat. Sb.}, {\bf 102}:4 (1977),
511--536 ({\it Russian}); English transl. in: {\it Math. USSR Sbornik}, {\bf 31} (1977), No. 4, 457 - 478.

\bibitem{Shikhov}  S.B.~Shikhov, {\it Problems in the Mathematical
Theory of Reactors. Linear Analysis}, Atomizdat, Moscow, 1973 ({\it
Russian}).

\bibitem{KNR} Yu. Kuperin, S. Naboko and R. Romanov,
"Spectral analysis of the transport operator: a functional model
approach", {\it Indiana Univ. Math. J.} \ {\bf 51}(2002), No. 6,
1389 - 1425.

\bibitem{HJ} V. Havin and B. J\"{o}riсke, {\it The uncertainty principle in harmonic analysis},
Springer-Verlag, Berlin, 1994.

\bibitem{kargaev} P. Kargaev, "The Fourier transform of the characteristic function of a set, vanishing on an interval", {\it Mat. Sb.}, {\bf
117(159)}:3 (1982), 397--411 ({\it Russian}); English transl. in:
{\it Math. USSR Sbornik}, {\bf 45}:3 (1983), 397–410. 
\end{thebibliography}
\end{document}